\documentclass[10pt]{amsart}
\usepackage[english]{babel}
\usepackage{amssymb}
\usepackage{enumitem}
\usepackage{hyperref}
\usepackage[initials]{amsrefs}
\usepackage[all]{xy}
\SelectTips{cm}{}

\newcommand{\bbC}{{\mathbb C}}

\newcommand{\pr}{\operatorname{pr}}

\newcommand{\Prob}{\operatorname{Prob}}

\newcommand{\SL}{\operatorname{SL}}

\newcommand{\Aut}{\operatorname{Aut}}

\newcommand{\inn}{\operatorname{inn}}

\newtheorem{theorem}{Theorem}[section]
\newtheorem{lemma}[theorem]{Lemma}

\newtheorem{cor}[theorem]{Corollary}

\newtheorem{prop}[theorem]{Proposition}

\theoremstyle{definition}

\newtheorem{defn}[theorem]{Definition}

\newtheorem{example}[theorem]{Example}

\newtheorem{remark}[theorem]{Remark}

\numberwithin{equation}{section}

\begin{document}

\title{An extension of Margulis' Super-Rigidity Theorem}

\author{Uri Bader}
\address{Weizmann Institute, Rehovot, Israel}
\email{bader@weizmann.ac.il}

\author{Alex Furman}
\address{University of Illinois at Chicago, Chicago, USA}
\email{furman@uic.edu}

\thanks{U.B. and A.F. were supported in part by the BSF grant 2008267.}
\thanks{U.B was supported in part by the ISF grant 704/08.}
\thanks{A.F. was supported in part by the NSF grant DMS 1611765.}


\maketitle

\begin{center}
	{\it To Gregory Margulis with gratitude and admiration.}
\end{center}

\begin{abstract}
We give an extension of Margulis' Super-Rigidity for higher rank lattices.
In our approach the target group could be defined over any complete valued field.
Our proof is based on the notion of Algebraic Representation of Ergodic Actions.
\end{abstract}

\section{Introduction}

In this note we present a proof of Margulis' Super-Rigidity Theorem \cite[Theorem~VII.5.6]{margulis-book} 
with algebraic target groups defined over valued fields that are not necessarily local.

\begin{theorem}[Margulis' super-rigidity for arbitrary target fields] \label{thm:margulis}\hfill{}\\
	Let $\ell$ be a local field, $H=\mathbf{H}(\ell)$ be the locally compact group
	formed by the $\ell$-points of a connected, semi-simple, algebraic group defined over $\ell$.
	Assume that the $\ell$-rank of $\mathbf{H}$ is at least two.
	Let $\Gamma<H$ be a lattice, and assume that the projection of $\Gamma$ in $H/N$ is non-discrete, 
	whenever $N\lhd H$ is the $\ell$-points of a proper normal $\ell$-isotropic subgroup.

	Let $k$ be a field with an absolute value, so that as a metric space $k$ is complete.
	Let $G=\mathbf{G}(k)$ be the $k$-points of a connected, adjoint, $k$-simple, algebraic group $\mathbf{G}$ defined over $k$.
	Let $\rho:\Gamma \to G$ be a homomorphism, and 
	assume that $\rho(\Gamma)$ is Zariski dense and unbounded in $G$.
	Then there exists a unique continuous homomorphism $\hat\rho:H\to G$
	such that $\rho=\hat\rho|_{\Gamma}$.
\end{theorem}

Note that the homomorphism $\hat\rho$ appearing in Theorem~\ref{thm:margulis} is necessarily given by some algebraic data - 
this is properly explained in Corollary~\ref{cor:finer}.
Theorem~\ref{thm:margulis} has a generalization to the so called $S$-arithmetic case, in that case the group $H$ 
is assumed to be a product of semi-simple algebraic groups over different local fields.
However, this generalization already follows from our result regarding super-rigidity for irreducible lattices in product of 
general locally compact groups \cite[Theorem~1.2]{BF-products} (alternatively: from \cite{Monod} or \cite{GKM}), so we do not discuss it here.
Likewise, in proving Theorem~\ref{thm:margulis} the case where $H$ has more than one non-compact factor follows essentially from \cite[Theorem~1.2]{BF-products}.
Thus, our main concern in this note is the case where $H$ has a unique non-compact factor.
This case will follow from Theorem~\ref{thm:main} below.
Before stating this theorem we will present two properties of topological groups.
\begin{itemize}
	\item[\textbf{(A)}] We say that a topological group $S$ satisfies condition (A) if every continuous, 
		isometric $S$-action without global fixed points on a metric space is topologically proper.
	\item[\textbf{(B)}] We say that a topological group $S$ satisfies condition (B) if $S$ 
		is topologically generated by closed non-compact subgroups $T_0,\ldots,T_n$
		such that, in a cyclic order, for every $i$, $T_{i+1}$ normalizes $T_{i}$ and at least one of the $T_i$'s is amenable.
\end{itemize}

\begin{example} \label{ex:sc}
Let $\ell$ be a local field.
Let $H=\mathbf{H}(\ell)$ be the $\ell$-points of a connected almost-simple algebraic group $\mathbf{H}$ defined over $\ell$.
If $\mathbf{H}$ is simply connected then, by \cite[Theorem~6.1]{Bader-Gelander}, $H$ satisfies condition (A).
Furthermore, if the $\ell$-rank of $\mathbf{H}$ is at least two, then $H$ satisfies condition (B).
In fact, the sequence of subgroups $T_i$ could be chosen from the root groups, properly ordered.
For example, for $H=\SL_3(\ell)$, we can use the sequence of subgroups
\[
	\begin{bmatrix}
	    1 & * & \\
	       & 1 & \\
	       &    & 1
	\end{bmatrix},\ \ 
	\begin{bmatrix}
	    1 & & * \\
	       & 1 & \\
	       &    & 1
	\end{bmatrix},\ \ 
	\begin{bmatrix}
	    1 & &   \\
	       & 1 & *\\
	       &    & 1
	\end{bmatrix},\ \ 
	\begin{bmatrix}
	    1 & &  \\
	    *  & 1 & \\
	       &    & 1
	\end{bmatrix},\ \ 
	\begin{bmatrix}
	    1 & &   \\
	       & 1 & \\
	    *  &    & 1
	\end{bmatrix},\ \ 
	\begin{bmatrix}
	    1 & &   \\
	       & 1 & \\
	       & * & 1
	\end{bmatrix}.
\]
\end{example}

\begin{theorem} \label{thm:main}\hfill{}\\
	Let $S$ be a second countable locally compact topological group and let $\Gamma<S$ be a lattice.
	Assume $S$ satisfies conditions (A) and (B).

	Let $k$ be a field with an absolute value, so that as a metric space $k$ is complete.
	Let $G=\mathbf{G}(k)$ be the $k$-points of a connected, adjoint, $k$-simple algebraic group $\mathbf{G}$ defined over $k$.
	Let $\rho:\Gamma \to G$ be a homomorphism, and 
	assume that $\rho(\Gamma)$ is Zariski dense and unbounded in $G$.
	Then there exists a unique continuous homomorphism $\hat\rho:S\to G$
	such that $\rho=\hat\rho|_{\Gamma}$.
\end{theorem}

The proof of Theorem~\ref{thm:main} will be given in \S\ref{sec:main}
and the detailed reduction of Theorem~\ref{thm:margulis} to Theorem~\ref{thm:main} will be carried out in \S\ref{sec:margulis}.
Currently, we do not know examples of locally compact groups satisfying conditions 
(A) and (B) which are not, essentially, higher rank semi-simple groups over local fields.
Nevertheless we find the formulation of Theorem~\ref{thm:main} useful not only for its potential applications, 
but also for psychological reasons, as it clarifies the different role played by the topological group $S$ and the algebraic group $G$.

\subsection*{Acknowledgments and disclaimers}
The content of this note is essentially contained in our manuscript \cite{AREA},
which we do not intend to publish as, in retrospect, we find it hard to read.
In our presentation here we do rely on \cite{BF-products} which contains other parts of the content of \cite{AREA}.
The manuscript \cite{AREA} contains further results, regarding general cocycle super-rigidity a la Zimmer, on which we intend to elaborate in a forthcoming paper.
We also rely here on the foundational work done in \cite{BDL}.
In our discussion in \S\ref{sec:explain} and in the reduction of Theorem~\ref{thm:margulis} to Theorem~\ref{thm:main} given in \S\ref{sec:margulis}
we rely heavily on the work of Borel and Tits, which we refer to via \cite{margulis-book}.

It is our pleasure to thank Bruno Duchesne and
Jean L\'{e}cureux
for their contribution to this project.
We are grateful to Michael Puschnigg for spotting an inaccuracy in the definition of a morphism of $T$-algebraic representations in an early draft of \cite{AREA}.
We would also like to thank Tsachik Gelander for numerous discussions.
Above all, we owe a huge mathematical debt to Gregory Margulis, whose incredible insight is reflected everywhere in this work.


\section{Ergodic theoretical preliminaries}

In this section we set our ergodic theoretical framework and notations.
Recall that a \emph{Polish space} is a topological space that is homeomorphic to a complete separable metric space. 
By a \emph{measurable space} we mean a set endowed with a $\sigma$-algebra.
A \emph{standard Borel space} is a measurable space that admits a measurable bijection to a Polish topological space, equipped
with the $\sigma$-algebra generated by its topology.
A \emph{Lebesgue space} is a standard Borel space endowed with the measure class of a probability measure and the completion of the Borel $\sigma$-algebra obtained by 
adding all subsets of null sets.
For a Lebesgue space $X$ we denote by $L^1(X)$, $L^2(X)$ and $L^\infty(X)$
the Banach spaces of equivalence classes of integrable, square-integrable and bounded measurable functions $X\to \bbC$, respectively; 
where two functions are equivalent if they agree a.e. 
We will also consider the space $L^0(X)$ consisting of equivalence classes of all measurable functions $X\to \bbC$.
Endowing this space with the topology of convergence in measure, this is a Polish topological space.

Every coset space of a locally compact second countable group is a Lebesgue space when endowed with its Haar measure class.
Unless otherwise stated, we will always regard the Haar measure class when considering locally compact second countable groups or their coset spaces as Lebesgue spaces.
Given a locally compact second countable group $S$, a \emph{Lebesgue $S$-space} is a Lebesgue space $X$ endowed with a measurable and measure class preserving action of $S$. 

Let $S$ be a locally compact second countable group and $X$ be a Lebesgue $S$-space.
Then $S$ acts on $L^\infty(X)$ via $sf(x)=f(s^{-1}x)$.
This $S$-action is isometric, but in general it is not continuous.
However, it is continuous when $L^\infty(X)$ is taken with the weak$^*$ topology induced by $L^1(X)$.
The action of $S$ on $X$ is said to be \emph{ergodic} if the only $S$-invariant function classes in $L^\infty(X)$ are the constant ones.

If the $S$-action preserves a finite measure in the given measure class on $X$, we say that the $S$-action is \emph{finite measure preserving}.
In such a case, the $S$-isometric action on $L^\infty(X)$ extends to an $S$-isometric action on $L^2(X)$,
which is norm continuous, hence unitary.
For finite measure preserving actions, 
the $S$-action on $X$ is ergodic if and only if the only invariant function classes in $L^2(X)$ are the constant ones.

The action of $S$ on $X$ is said to be \emph{metrically ergodic} if
for every separable metric space $(U,d)$ on which $S$ acts continuously by isometries,
any a.e defined $S$-equivariant map $\phi:X\to U$ is essentially constant.

\begin{example} \label{ex:ME}
Let $S$ be a locally compact second countable group satisfying condition~(A) and let $T<S$
be a non-compact closed subgroup.
Then the action of $S$ on $S/T$ is metrically ergodic.
\end{example}

Recall that a finite measure preserving action of $S$ on $X$ is \emph{weakly mixing} if the diagonal $S$-action on $X\times X$ is ergodic.
Weak mixing is equivalent to the condition that the only $S$-invariant finite dimensional 
subspace of $L^2(X)$ is the constant functions.
For finite measure preserving actions metric ergodicity is equivalent to weak mixing (cf. \cite{GW}*{Theorem 2.1}).

\begin{lemma} \label{lem:TWM}
Let $S$ be a locally compact second countable group satisfying condition~(A) and let $X$ be a finite measure preserving $S$-ergodic Lebesgue space.
Then for every non-compact closed subgroup $T<S$, the restricted $T$-action on $X$ is weakly mixing.
\end{lemma}

\begin{proof}
The isometric action of $S$ on the unit sphere $U$ of the orthogonal complement of the constant functions in $L^2(X)$ is continuous
and has no fixed points, by the ergodicity assumption. Therefore the $S$-action on $U$ is proper. 
It follows that there is no $T$-invariant compact subset in $U$.
Therefore $T$ has no finite dimensional subrepresentations in $L^2(X)$ except for the constant functions.
\end{proof}

The action of $S$ on $X$ is said to be \emph{amenable} if for every $S$-Borel space $V$
and an essentially surjective $S$-equivariant Borel map $\pi: V\to X$ with compact convex fibers,
such that the $S$-action restricted to the fibers is by continuous affine maps,
one has an a.e defined $S$-invariant measurable section (see \cite[Definition~4.3.1]{zimmer-book}).

\begin{example}[\cite{zimmer-book}*{Proposition~4.3.2}] \label{ex:amen}
Let $T<S$ be an amenable closed subgroup.
Then the action of $S$ on $S/T$ is amenable.
\end{example}

\section{Algebraic representation of ergodic actions} \label{sec:AREA}

In this section we fix a field $k$ with a non-trivial absolute value which is separable and complete
(as a metric space)
and a $k$-algebraic group $\mathbf{G}$.
We note that $\mathbf{G}(k)$, when endowed with the $k$-analytic topology, is a Polish topological group, see \cite{BDL}*{Proposition~2.2}.
We also fix a locally compact second countable group $T$ and a homomorphism $\tau:T\to \mathbf{G}(k)$
which is continuous, considering $\mathbf{G}(k)$ with its analytic topology.
Let us also fix a Lebesgue $T$-space $X$.

\begin{defn}
Given all the data above, an \emph{algebraic representation} of $X$
consists of a $k$-$\mathbf{G}$-algebraic variety $\mathbf{V}$ and
an a.e defined measurable map $\phi:X \to \mathbf{V}(k)$ such that for every $t\in T$ and for a.e. $x\in X$
\[ 
	\phi(tx)=\tau(t)\phi(x). 
\]
We shall say that $\mathbf{V}$ \emph{is an algebraic representation of} $X$,
and denote $\phi$ by $\phi_\mathbf{V}$ for clarity.
A \emph{morphism} from the algebraic representation $\mathbf{U}$ to the algebraic representation $\mathbf{V}$ consists of
a $\mathbf{G}$-equivariant $k$-morphism $\pi:\mathbf{U}\to \mathbf{V}$ 
such that $\phi_\mathbf{V}$ agrees almost everywhere with $\pi\circ \phi_\mathbf{U}$.
An algebraic representation $\mathbf{V}$ of $X$ is said to be a \emph{coset algebraic representation}
if in addition
$\mathbf{V}$ is isomorphic as an algebraic representation to a coset variety $\mathbf{G}/\mathbf{H}$ for some $k$-algebraic subgroup $\mathbf{H}<\mathbf{G}$.
\end{defn}

Ergodic properties of $X$ are reflected in its category of algebraic representations.

\begin{prop}[{\cite[Proposition 4.2]{BF-products}}] \label{prop:AG-ergodic}
	Assume $X$ is $T$-ergodic.
	Then for every algebraic representation $\phi_\mathbf{V}:X\to \mathbf{V}(k)$ there exists a coset representation $\phi_{\mathbf{G}/\mathbf{H}}:X \to\mathbf{G}/\mathbf{H}(k)$
and a morphism of algebraic representations $\pi:{\bf G}/{\bf H}\to {\bf V}$,
that is a $\bf G$-equivariant $k$-morphism $\pi$ such that for a.e $x\in X$, $\phi_{\bf V}(x)=\pi\circ \phi_{{\bf G}/{\bf H}}(x)$.
\end{prop}

In case the $T$-action on $X$ is weakly mixing (and in particular, finite measure preserving) the category of representation of $X$ is essentially trivial.

\begin{prop} \label{prop:AG-WM}
Assume $X$ is $T$-weakly mixing.
Then for every algebraic representation $\phi:X\to \mathbf{V}(k)$, $\phi$ is essentially constant.
Further, if $\tau(T)$ is Zariski dense in $\mathbf{G}$ then the essential image of $\phi$ is $\mathbf{G}$-invariant.
\end{prop}

\begin{proof}
Letting $\mu\in\Prob(\mathbf{V}(k))$ be the push forward by $\phi$ of the measure on $X$ and $L$ be the closure of $\tau(T)$ in $\mathbf{G}(k)$,
it follows from \cite[Corollary 1.13]{BDL} that $\phi$ is essentially constant and its essential image is $L$-fixed.
\end{proof}

The following theorem guarantees non-triviality of the category of representations of $X$
(here, the $T$-action on $X$ is not assumed to preserve a finite measure).

\begin{theorem}[{\cite[Theorem 4.5]{BF-products}, \cite[Theorem 1.17]{BDL}}] \label{prop:AG-ME+Am}
Assume the $T$-Lebesgue space $X$ is both amenable and metrically ergodic.
Assume the $k$-algebraic group $\mathbf{G}$ is connected, $k$-simple and adjoint
and assume that $\tau(T)$ is Zariski dense and unbounded in $\mathbf{G}(k)$.
Then there exists a coset representation $\phi:X\to \mathbf{G}/\mathbf{H}(k)$ for some proper $k$-subgroup $\mathbf{H}\lneq \mathbf{G}$.
\end{theorem}

\section{$T$-algebraic representations of $S$}

Throughout this section we fix a locally compact second countable group $S$ and a lattice $\Gamma<S$.
We endow $S$ with its Haar measure and regard it as a Lebesgue space.
We also fix a field $k$ endowed with a non-trivial absolute value which is separable and complete (as a metric space)
and a $k$-algebraic group $\mathbf{G}$.
We denote by $G$ the Polish group ${\bf{G}}(k)$.
Finally, we fix a homomorphism $\rho:\Gamma\to G$.

\begin{defn}
Given all the data above, for a closed subgroup $T<S$, a $T$-\emph{algebraic representation of} $S$
consists of the following data:
\begin{itemize}
	\item a $k$-algebraic group $\mathbf{L}$,
	\item a $k$-$(\mathbf{G}\times \mathbf{L})$-algebraic variety $\mathbf{V}$, regarded as a left $\bf G$, right $\bf L$ space,
	on which the $\mathbf{L}$-action is faithful,
	\item a homomorphism $\tau:T\to \mathbf{L}(k)$ with a Zariski dense image,
	\item an associated algebraic representation of the $\Gamma\times T$-space $S$ on $\mathbf{V}$,
	where $\Gamma$ acts on the left, $T$ acts on the right of the Lebesgue space $S$.
	That is, a Haar a.e defined measurable map $\phi:S \to \mathbf{V}(k)$ such that for almost every $s\in S$, every $\gamma\in \Gamma$ and every $t\in T$,
	\[ 
		\phi(\gamma s t)=\rho(\gamma) \phi(s)\tau(t).
	\]
\end{itemize}
We abbreviate the notation by saying that $\mathbf{V}$ is a $T$-algebraic representation of $S$,
denoting the extra data by $\mathbf{L}_\mathbf{V}, \tau_\mathbf{V}$ and $\phi_\mathbf{V}$.
Given another $T$-algebraic representation $\mathbf{U}$ we let $\mathbf{L}_{\mathbf{U},\mathbf{V}}<\mathbf{L}_\mathbf{U}\times \mathbf{L}_\mathbf{V}$ be the Zariski closure of the image 
of $\tau_\mathbf{U}\times \tau_\mathbf{V}:T\to \mathbf{L}_\mathbf{U}\times \mathbf{L}_\mathbf{V}$. 
Note that $\mathbf{L}_{\mathbf{U},\mathbf{V}}$ acts on $\mathbf{U}$ and $\mathbf{V}$ 
via its projections to $\mathbf{L}_\mathbf{U}$ and $\mathbf{L}_\mathbf{V}$ correspondingly.
A \emph{morphism} of $T$-algebraic representations of $S$ from the $T$-algebraic representation $\mathbf{U}$ to the $T$-algebraic representation $\mathbf{V}$ is
a $\mathbf{G}\times\mathbf{L}_{\mathbf{U},\mathbf{V}}$-equivariant $k$-morphism $\pi:\mathbf{U}\to \mathbf{V}$
such that $\phi_\mathbf{V}$ agrees a.e with $\pi\circ \phi_\mathbf{U}$.
\end{defn}

Fix a $k$-subgroup $\mathbf{H}<\mathbf{G}$ and denote $\mathbf{N}=N_\mathbf{G}(\mathbf{H})$.
This is again a $k$-subgroup.
Any element $n\in \mathbf{N}$ gives a $\mathbf{G}$-automorphism of $\mathbf{G}/\mathbf{H}$ by
$g\mathbf{H}\mapsto gn^{-1}\mathbf{H}$.
It is easy to see that the homomorphism $\mathbf{N} \to \Aut_\mathbf{G}(\mathbf{G}/\mathbf{H})$ thus obtained is surjective and its kernel is $\mathbf{H}$.
Under the obtained identification $\mathbf{N}/\mathbf{H} \cong \Aut_\mathbf{G}(\mathbf{G}/\mathbf{H})$,
the $k$-points of the $k$-group $\mathbf{N}/\mathbf{H}$ are identified with the $k$-$\mathbf{G}$-automorphisms of $\mathbf{G}/\mathbf{H}$.

\begin{defn}
	A $T$-algebraic representation of $S$ is said to be a \emph{coset $T$-algebraic representation} if it is isomorphic as a $T$-algebraic representation
	to $\mathbf{G}/\mathbf{H}$ for some $k$-algebraic subgroup $\mathbf{H}<\mathbf{G}$,
	and $\mathbf{L}$ corresponds to a $k$-subgroup of $N_\mathbf{G}(\mathbf{H})/\mathbf{H}$ which acts on $\mathbf{G}/\mathbf{H}$ as described above.
\end{defn}

It is clear that the collection of $T$-algebraic representations of $S$ and their morphisms form a category.

\begin{theorem} \label{thm:initial}
Assume the $T$-action on $S/\Gamma$ is weakly mixing.
Then the category of $T$-algebraic representations of $S$ has an initial object
and this initial object is a coset $T$-algebraic representation.
\end{theorem}

We will first prove the following lemma.

\begin{lemma} \label{lem:WM contraction}
	Assume the $T$-action on $S/\Gamma$ is weakly mixing.
	Let $\mathbf{V}$ be a $T$-algebraic representation of $S$.
	Then there exists a coset $T$-algebraic representation of $S$ for some $k$-algebraic subgroup $\mathbf{H}<\mathbf{G}$
	and a morphism of $T$-algebraic representations $\pi:\mathbf{G}/\mathbf{H}\to \mathbf{V}$.
\end{lemma}

\begin{proof}
	The $T$-action on $S/\Gamma$ is weakly mixing, and in particular ergodic, thus the $\Gamma\times T$ -action on $S$ is ergodic.
	Applying Proposition~\ref{prop:AG-ergodic} we get that there exists a coset representation $(\mathbf{G}\times \mathbf{L})/{\bf M}$ for some $k$-algebraic subgroup
	${\bf M}<\mathbf{G}\times \mathbf{L}$ and a morphism of algebraic representations $\pi:(\mathbf{G}\times \mathbf{L})/{\bf M}\to \mathbf{V}$.
	We are thus reduced to the case
	$\mathbf{V}=(\mathbf{G}\times \mathbf{L})/{\bf M}$.
	Denote the obvious projection from $\mathbf{G} \times \mathbf{L}$
	to $\mathbf{G}$ and $\mathbf{L}$ correspondingly by $\pr_1$ and $\pr_2$.
	The composition of the map $\phi:S \to (\mathbf{G}\times \mathbf{L})/\mathbf{M}(k)$ with the $\mathbf{G}(k)$-invariant map 
	\[
		(\mathbf{G}\times \mathbf{L})/\mathbf{M}(k)\to \mathbf{L}/\pr_2(\mathbf{M})(k)
	\]
	clearly factors through $S/\Gamma$, and thus gives a coset representation of the $T$-Lebesgue space $X=S/\Gamma$ on $\mathbf{L}/\pr_2(\mathbf{M})$.
	Applying Proposition~\ref{prop:AG-WM}, we conclude that $\mathbf{L}/\pr_2(\mathbf{M})$ contains a $\mathbf{G}$-invariant point, thus $\pr_2(\mathbf{M})=\mathbf{L}$.
	It follows that as $\mathbf{G}$-varieties, $(\mathbf{G}\times \mathbf{L})/\mathbf{M} \cong \mathbf{G}/\pr_1({\bf M})$
	and letting $\mathbf{H}=\pr_1(\mathbf{M})<\mathbf{G}$, the lemma follows.
\end{proof}

\begin{proof}[Proof of Theorem~\ref{thm:initial}]
We consider the collection
\[ 
	\{\mathbf{H}<\mathbf{G}~|~\mathbf{H}\mbox{ is defined over } k \mbox{ and there exists a coset $T$-representation to } \mathbf{G}/\mathbf{H} \}. 
\]
This is a non-empty collection as it contains $\mathbf{G}$.
By the Neotherian property, this collection contains a minimal element.
We choose such a minimal element $\mathbf{H}_0$
and fix corresponding algebraic $k$-subgroup $\mathbf{L}_0<N_\mathbf{G}(\mathbf{H}_0)/\mathbf{H}_0$,
homomorphism $\tau_0:T \to \mathbf{L}_0(k)$
and a representation $\phi_0:S \to (\mathbf{G}/\mathbf{H}_0)(k)$.
We argue to show that this coset $T$-representation is the required initial object.

Fix any $T$-algebraic representation of $S$, $\mathbf{V}$.
It is clear that, if exists, a morphism of $T$-algebraic representations from $\mathbf{G}/\mathbf{H}_0$ to $\mathbf{V}$ is unique, as two different $\mathbf{G}$-maps
$\mathbf{G}/\mathbf{H}_0\to \mathbf{V}$ agree nowhere.
We are left to show existence.
To this end we consider
the product $T$-algebraic representation $\mathbf{V}\times \mathbf{G}/\mathbf{H}_0$ given by the data $\phi=\phi_\mathbf{V}\times \phi_0$,
$\tau=\tau_\mathbf{V}\times \tau_0$ and $\mathbf{L}$ being he Zariski closure of $\tau(T)$ in $\mathbf{L}_\mathbf{V}\times \mathbf{L}_0$.
Applying Lemma~\ref{lem:WM contraction} to this product $T$-algebraic representation we obtain the following commutative diagram
\[
	\xymatrix{ 
		S \ar@{.>}[r] \ar[d]^{\phi_\mathbf{V}} \ar[rd]^{\phi} \ar@/^3pc/[rrd]^{\phi_0} & \mathbf{G}/\mathbf{H}(k) \ar@{.>}[d]_{\pi} &  \\
		   \mathbf{V}(k) & \mathbf{V}\times \mathbf{G}/\mathbf{H}_0(k) \ar[r]^{~~\pr_2} \ar[l]_{\pr_1~~~~} & \mathbf{G}/\mathbf{H}_0(k)  
	}
\]
By the minimality of $\mathbf{H}_0$, the $\mathbf{G}$-morphism $\pr_2\circ \pi:\mathbf{G}/\mathbf{H} \to \mathbf{G}/\mathbf{H}_0$ must be a $k$-isomorphism,
hence an isomorphism of $T$-algebraic representations.
We thus obtain the morphism of $T$-algebraic representations
\[ 
	\pr_1\circ \pi \circ (\pr_2\circ \pi)^{-1}:\mathbf{G}/\mathbf{H}_0(k) \to \mathbf{V}(k). 
\]
This completes the proof of Theorem~\ref{thm:initial}.
\end{proof}

\begin{remark} \label{rem:initial}
	Let the data $\mathbf{G}/\mathbf{H}$, $\tau:T \to \mathbf{L}(k)<N_\mathbf{G}(\mathbf{H})/\mathbf{H}(k)$ 
	and $\phi:S \to (\mathbf{G}/\mathbf{H})(k)$ form an initial object
	in the category of $T$-algebraic representations of $S$.
	For $g\in \mathbf{G}(k)$ we get a $\mathbf{G}$-equivariant $k$-isomorphism $\pi_g:\mathbf{G}/\mathbf{H}\to \mathbf{G}/\mathbf{H}^g$ given by 
	$x\mathbf{H}\mapsto xg^{-1}\mathbf{H}^g$.
	Denoting by $\inn(g):N_\mathbf{G}(\mathbf{H})/\mathbf{H}\to N_\mathbf{G}(\mathbf{H}^g)/\mathbf{H}^g$ the $k$-isomorphism $n\mathbf{H}\mapsto n^g\mathbf{H}^g$
	and by $\mathbf{L}^g$ the image of $\mathbf{L}$ under $\inn(g)$
	we get that the data
	\[
		\mathbf{G}/\mathbf{H}^g,\quad 
		\inn(g)\circ \tau:T \to \mathbf{L}^g(k)<N_\mathbf{G}(\mathbf{H}^g)/\mathbf{H}^g(k),\quad 
		\pi_g\circ \phi:S \to \mathbf{G}/\mathbf{H}(k)
	\]
	forms another $T$-algebraic coset representations of $S$, isomorphic to the one given above,
	thus again an initial object
	in the category of $T$-algebraic representations of $S$.
	Furthermore, it is easy to verify that any actual coset presentation of the initial object in the category of $T$-algebraic representations of $S$
	is of the above form, for some $g\in \mathbf{G}(k)$.
\end{remark}

It turns out that an initial object in the category of $T$-algebraic representations of $S$ extends naturally to an $N$-algebraic representation of $S$, where $N$ denotes the normalizer of $T$ in $S$.

\begin{theorem} \label{thm:normalizer}
	Assume the action of $T$ on $S/\Gamma$ is weakly mixing and let $\mathbf{G}/\mathbf{H}$,
	$\tau:T \to \mathbf{L}(k)<N_\mathbf{G}(\mathbf{H})/\mathbf{H}(k)$ and $\phi:S \to \mathbf{G}/\mathbf{H}(k)$ be an initial object
	in the category of $T$-algebraic representations of $S$, as guaranteed by Theorem~\ref{thm:initial}.
	Then the map $\tau:T \to N_\mathbf{G}(\mathbf{H})/\mathbf{H}(k)$ extends to the normalizer $N=N_S(T)$ of $T$ in $S$,
	and the map $\phi$ could be seen as an $N$-algebraic representations of $S$.
	More precisely, there exists a continuous homomorphism $\bar{\tau}:N\to N_\mathbf{G}(\mathbf{H})/\mathbf{H}(k)$ satisfying $\bar{\tau}|_T=\tau$
	such that, denoting by $\bar{\mathbf{L}}$ the Zariski closure of $\bar{\tau}(N)$ in $N_\mathbf{G}(\mathbf{H})/\mathbf{H}$, the data
	\[
		\mathbf{G}/\mathbf{H},\quad
		\bar{\tau}:N \to \bar{\mathbf{L}}(k)<N_\mathbf{G}(\mathbf{H})/\mathbf{H}(k),\quad
		\phi:S \to \mathbf{G}/\mathbf{H}(k)
	\] 
	forms an $N$-algebraic coset representation.
	Moreover, this $N$-algebraic coset representation is an initial object in the category of $N$-algebraic representations.
\end{theorem}

\begin{proof}
Fix $n\in N$.
Set $\tau'=\tau\circ \inn(n):T\to \mathbf{L}(k)$, where $\inn(n):T\to T$ denotes the inner automorphism $t\mapsto t^n=ntn^{-1}$,
and $\phi'=\phi\circ R_n:S\to \mathbf{G}/\mathbf{H}(k)$, where $R_n:S\to S$ denotes the right regular action $s\mapsto sn^{-1}$.
We claim that the data $\mathbf{L}$, $\mathbf{G}/\mathbf{H}$, $\tau'$ and $\phi'$ forms a new $T$-representation of $S$.
Indeed, for almost every $s\in S$, every $\gamma\in \Gamma$ and every $t\in T$,
\[ 
	\phi'(\gamma s t)=\phi(\gamma s t n^{-1})=\phi(\gamma s n^{-1} t^n)= \rho(\gamma) \phi' (s)\tau'(t). 
\]
Since the $T$-algebraic representation of $S$ given by $\mathbf{L}$, $\mathbf{G}/\mathbf{H}$,
$\tau$ and $\phi$, forms an initial object, we get the dashed vertical arrow, which we denote $\bar{\tau}(n)$, in the following diagram:
\[
	\xymatrix{ S \ar[r]^{\phi} \ar[rd]_{\phi'} & \mathbf{G}/\mathbf{H}(k) \ar@{.>}[d]^{\bar{\tau}(n)}  \\
		    & \mathbf{G}/\mathbf{H}(k)  }
\]
It follows from the uniqueness of the dashed arrow, that the map $n\mapsto \bar{\tau}(n)$ is a homomorphism from $N$ to the
the group of $k$-$\mathbf{G}$-automorphism of $\mathbf{G}/\mathbf{H}$,
which we identify with $N_\mathbf{G}(\mathbf{H})/\mathbf{H}(k)$.
For $n\in T$ the map $\tau(n):\mathbf{G}/\mathbf{H}\to \mathbf{G}/\mathbf{H}$ could also be taken to be the dashed arrow, thus $\bar{\tau}|_{T}=\tau$, by uniqueness.
The fact that the homomorphism $\bar{\tau}:N\to N_\mathbf{G}(\mathbf{H})/\mathbf{H}(k)$ 
is necessarily continuous is explained in the proof of \cite{BF-products}*{Theorem~4.7}.
We define $\bar{\mathbf{L}}$ to be the Zariski closure of $\bar{\tau}(N)$ in $N_\mathbf{G}(\mathbf{H})/\mathbf{H}$.
We thus indeed obtain an $N$-representation of $S$, given by the algebraic group $\bar{\mathbf{L}}$, the variety $\mathbf{G}/\mathbf{H}$, 
the homomorphism $\bar{\tau}:N \to \bar{\mathbf{L}}(k)$ and the (same old) map
$\phi:S \to \mathbf{G}/\mathbf{H}(k)$.

The final part, showing that the above data forms an initial object in the category of $N$-algebraic representations of $S$,
is left to the reader, as we will not use this fact in the sequel.
\end{proof}

\begin{cor} \label{cor:conj}
Let $T_1,T_2<S$ be closed subgroups and assume for each $i\in \{1,2\}$ the action of $T_i$ on $S/\Gamma$ is weakly mixing, and let 
\[
	\mathbf{G}/\mathbf{H}_i,\quad \tau_i:T_i \to \mathbf{L}_i(k)<N_\mathbf{G}(\mathbf{H}_i)/\mathbf{H}_i(k),\quad
	\phi_i:S \to (\mathbf{G}/\mathbf{H}_i)(k)
\] 
be the corresponding initial objects in the categories of $T_i$-algebraic representations of $S$.
Assume that $T_2$ normalizes $T_1$.
Then a conjugate of $\mathbf{H}_2$ is contained in $\mathbf{H}_1$, and 
if $\mathbf{H}_2=\mathbf{H}_1$ then also $\phi_2=\phi_1$.
\end{cor}

\section{Proof of Theorem~\ref{thm:main}} \label{sec:main}

In the proof below we shall need the following general Lemma 
that allows to assemble a continuous homomorphism $\tau:S\to G$ 
from continuous homomorphisms of subgroups $\tau_i:T_i\to G$.

\begin{lemma}\label{lem:LTG}
	Let $S$ be a locally compact second countable group, $\{T_i<S\}_{i\in I}$ a countable family
	of closed subgroups that together topologically generate $S$.
	Let $G$ be a Polish topological group and for each $i\in I$ let $\tau_i:T_i\to G$ be a continuous homomorphism.
	
	Let $X$ be a Lebesgue $S$-space, and assume that there exists a single 
	measurable map $\phi:X\to G$ so that for every $i\in I$ and every $t\in T_i$ for a.e. $x\in X$
	\[
		\phi(tx)=\phi(x)\tau_i(t)^{-1}.
	\]
	Then there exists a continuous homomorphism $\tau:S\to G$ so that $\tau|_{T_i}=\tau_i$ and 
	\[
		\phi(sx)=\phi(x)\tau(s)^{-1}
	\]
	for every $s\in S$ and a.e. $x\in X$.
\end{lemma}
\begin{proof}
	Let us fix a probability measure $m$ in the given measure class on $X$ and
	consider the space $L^0(X,G)$ of (equivalence classes of) measurable functions $\psi:X\to G$ taken with the topology of convergence
	in measure: given an open neighborhood $U$ of $1\in G$ and $\epsilon>0$ a $(U,\epsilon)$-neighborhood of $\psi$ consists of
	classes of those measurable functions $\psi':X\to G$ for which $m\{x\in X\mid \psi'(x)\in \psi(x)U\}>1-\epsilon$.
	This topology is Polish. 
	
	The right translation action of $G$ on $L^0(X,G)$ given by $g:\psi(x)\mapsto \psi(x)g^{-1}$ is clearly free and continuous.
	In fact, every $G$-orbit is homeomorphic to $G$ and is closed in $L^0(X,G)$.
	To see this we assume $g_i\psi\to \psi'$ in measure and argue to show that there exists $g\in G$ such that $\psi'(x)=\psi(x)g^{-1}$.
The given converges in measure implies that there exists a subsequence $g_{i_j}$
	for which $\psi(x)g_{i_j}^{-1}\to \psi'(x)$ for a.e. $x\in X$. 
	Therefore $g_{i_j}\to \psi'(x)^{-1}\psi(x)$ for a.e. $x\in X$, 
	thus indeed there is a limit $g=\lim_{j\to\infty}g_{i_j}$ in $G$ such that $\psi'(x)=\psi(x)g^{-1}$.
	
	The group $S$ acts on $L^0(X,G)$ by precomposition $s:\psi(x)\mapsto\psi(s^{-1}x)$.
	This action is also continuous, because any measurable measure class preserving action $S\times X\to X$
	of a locally compact group has the property that given $\epsilon>0$, there exists $\delta>0$ and a neighborhood $V$ of $1\in S$ 
	so that for any measurable $E\subset X$ with $m(E)<\delta$ one has $m(sE)<\epsilon$ for every $s\in V$.
	
	Now consider $\phi\in L^0(X,G)$ as in the Lemma. By the assumption, for each $i\in I$ the $T_i$-orbit of $\phi$
	lies in the $G$-orbit of $\phi$. 
	Since $S$ is topologically generated by $\cup_I T_i$, the $S$-action is continuous and the $G$-orbit of $\phi$ is closed,
	it follows that the $S$-orbit of $\phi$ is contained in the $G$-orbit of $\phi$.
	Hence, for every $s\in S$ there is $\tau(s)\in G$ so that a.e. on $X$
	\[
		\phi(xs^{-1})=\phi(x)\tau(s)^{-1}.
	\]
	This defines a homomorphism $\tau:S\to G$ that extends all $\tau_i:T_i\to G$.
	Continuity of the homomorphism follows from the fact that the $S$-action is continuous and $G$ is homeomorphic to the $G$-orbit of $\phi$. 
\end{proof}

We now return to the proof of Theorem~\ref{thm:main}.
We let $\mathbf{G}$ be a connected, adjoint, $k$-simple algebraic group defined over $k$ as in Theorem~\ref{thm:main}.
By the fact that $\rho(\Gamma)$ 
is unbounded in ${\bf G}(k)$ we get that the given absolute value on $k$ is non-trivial.
Further, by the countability of $\Gamma$, we may replace $k$ with a complete and separable (in the topological sense) subfield $k'$ such that $\rho(\Gamma)\subset \mathbf{G}(k')$. 
We will therefore assume below that the given absolute value on the field $k$ is non-trivial and that $k$ is complete and separable as a metric space.
Accordingly, we will regard $G=\mathbf{G}(k)$ as a Polish group.
We let $T_0,\ldots,T_n<S$ be subgroups as guaranteed by condition (B)
and assume, as we may, that $T_0$ is amenable.

We fix $i\in \{0,\dots,n\}$.
By Lemma~\ref{lem:TWM} the action of $T_i$ on $S/\Gamma$ is weakly mixing,
thus by Theorem~\ref{thm:initial}
the category of $T_i$-algebraic representations of $S$ has an initial object which is a coset $T_i$-algebraic representation.
We denote by 
\[
	\mathbf{G}/\mathbf{H}_i,
	\quad 
	\tau_i:T_i \to \mathbf{L}_i(k)<N_\mathbf{G}(\mathbf{H}_i)/\mathbf{H}_i(k),
	\quad 
	\phi_i:S \to \mathbf{G}/\mathbf{H}_i(k)
\] 
the data forming this initial object.

By Examples~\ref{ex:ME} and \ref{ex:amen} the $S$-action on $S/T_0$ is metrically ergodic and amenable.
It follows from \cite[Lemma~3.5]{BF-products} that also the $\Gamma$-action on $S/T_0$ is metrically ergodic and amenable.
By Proposition~\ref{prop:AG-ME+Am} we get that there exists a coset representation $\phi':S/T_0\to \mathbf{G}/\mathbf{H}'(k)$ for some proper $k$-subgroup $\mathbf{H}'\lneq \mathbf{G}$.
Setting $\mathbf{L}'=\{e\}<N_\mathbf{G}(\mathbf{H}')/\mathbf{H}'$ and letting $\tau':T_0\to \mathbf{L}'(k)$ be the trivial homomorphism
we view $\phi'$ as a $T_0$-equivariant map from $S$ to $\mathbf{G}/\mathbf{H}'(k)$, thus getting a non-trivial $T_0$-algebraic coset representation of $S$.
It follows that $\mathbf{H}_0$ is a proper subgroup of $\mathbf{G}$.

By Corollary~\ref{cor:conj} for each $i\in \{0,\dots,n\}$ 
a conjugate of $\mathbf{H}_{i-1}$ is contained in $\mathbf{H}_i$,
where $i$ is taken in a cyclic order.
Going a full cycle we get that the groups $\mathbf{H}_i$ are all conjugated.
Using Remark~\ref{rem:initial} we assume that they all coincide and we denote this common group by $\mathbf{H}$.
In particular, we have $\mathbf{H}=\mathbf{H}_0 \lneq \mathbf{G}$.
Using again Corollary~\ref{cor:conj}, we obtain that the maps $\phi_i:S \to \mathbf{G}/\mathbf{H}(k)$ all coincide and we denote this common map by 
\[
	\phi:S\to \mathbf{G}/\mathbf{H}(k).
\]
Let $\mathbf{N}=N_\mathbf{G}(\mathbf{H})$ and $\mathbf{L}$ be the algebraic subgroup generated by 
$\mathbf{L}_0,\ldots,\mathbf{L}_n$ in $\mathbf{N}/\mathbf{H}$.
By \cite[2.1(b)]{borel}, $\mathbf{L}<\mathbf{N}/\mathbf{H}$ is a $k$-algebraic subgroup.
Denote by $\hat{\mathbf{L}}$ the preimage of $\mathbf{L}$
under the quotient map $\mathbf{N} \to \mathbf{N}/\mathbf{H}$
and note that $\hat{\mathbf{L}}<\mathbf{N}$ is a $k$-algebraic subgroup.
We conclude that the $k$-$\mathbf{G}$-morphism $\pi:\mathbf{G}/\mathbf{H} \to \mathbf{G}/\hat{\mathbf{L}}$
is $\mathbf{L}_i$-invariant for every $i\in \{0,\dots,n\}$.
It follows that $\pi\circ \phi:S \to \mathbf{G}/ \hat{\mathbf{L}}(k)$ is $T_i$-invariant for every $i$.
Since $S$ is topologically generated by the groups $T_i$,
we conclude that $\pi\circ \phi:S \to \mathbf{G}/\hat{\mathbf{L}}(k)$ is $S$-invariant.
Thus $\phi$ is a constant map and its essential image is a $\rho(\Gamma)$-invariant point in $\mathbf{G}/ \hat{\mathbf{L}}$.
Since $\rho(\Gamma)$ is Zariski dense in $\mathbf{G}$, this point is $\mathbf{G}$-invariant.
We conclude that $\hat{\mathbf{L}}=\mathbf{G}$.
Since $\hat{\mathbf{L}}<\mathbf{N}<\mathbf{G}=\hat{\mathbf{L}}$ we get that $\hat{\mathbf{L}}=\mathbf{N}=\mathbf{G}$.
Thus, $\mathbf{G}$ normalizes $ \mathbf{H}$.
Since $\mathbf{G}$ is $k$-simple and $\mathbf{H}\lneq \mathbf{G}$ we conclude that $\mathbf{H}$ is trivial.
In particular $\mathbf{L}=\hat{\mathbf{L}}=\mathbf{G}$ and it acts on $\mathbf{G}/\mathbf{H}=\mathbf{G}$ by right multiplication.

\medskip

To summarize:
We have an a.e defined measurable map $\phi:S\to \mathbf{G}(k)$
and for every $i\in\{1,\dots,n\}$ a continuous homomorphism $\tau_i:T_i\to \mathbf{G}(k)$
such that for every $\gamma\in \Gamma$, $t\in T_i$ for a.e $x\in S$
\begin{equation} \label{eq:tau-i}
	\phi(\gamma xt^{-1})=\rho(\gamma)\phi(x)\tau_i(t)^{-1}.
\end{equation}
We also have that the algebraic group generated by $\tau_1(T_1),\ldots, \tau_n(T_n)$ is $\mathbf{G}$.

Taking $\gamma=1$ in equation (\ref{eq:tau-i}) and applying Lemma~\ref{lem:LTG} for $X=S$
endowed with the $S$-action $s:x\mapsto xs^{-1}$,  
we get a continuous homomorphism $\tau:S\to \mathbf{G}(k)$ with $\tau|_{T_i}=\tau_i$ so that $\phi(xs^{-1})=\phi(x) \tau(s)^{-1}$
for every $s\in S$ and for a.e $x\in S$.
It follows now from equation (\ref{eq:tau-i}) that for every $\gamma\in \Gamma$, every $s\in S$ and for a.e $x\in S$
\begin{equation} \label{eq:eqD}
	\phi(\gamma xs^{-1})=\rho(\gamma)\phi(x) \tau(s)^{-1}.
\end{equation}
The a.e defined measurable map $\Phi:S\to G$ given by $\Phi(s)=\phi(s)\tau(s)^{-1}$ is $S$-invariant, hence essentially constant.
Denoting its essential image by $g\in \mathbf{G}(k)$ we get that for a.e $s\in S$,
$\phi(s)=g\tau(s)$.
Equation~(\ref{eq:eqD}) gives that for every $\gamma\in \Gamma$, $s\in S$ and for a.e $x\in S$
\[ 
	g\tau(\gamma xs)=\rho(\gamma)g\tau(x) \tau(s). 
\]
As the above is an a.e satisfied equation of continuous functions in the parameter $s$, it is satisfied everywhere.
Taking $x=s=e$, we get that for every $\gamma\in \Gamma$,
\[ 
	g\tau(\gamma)=\rho(\gamma)g. 
\]
Setting $\hat\rho(s)=g\tau(s)g^{-1}$ we get that $\hat\rho:S\to G$ is a continuous homomorphism such that $\hat\rho|_\Gamma=\rho$.
The uniqueness of $\hat\rho$ follows from \cite[Lemma 6.3]{BF-products} and the proof of Theorem~\ref{thm:main} is completed.


\section{Continuous homomorphisms of algebraic groups} \label{sec:explain}

\begin{lemma} \label{lem:extscal}
Let $k$ be a complete valued field and $k'$ a finite field extension,
endowed with the extended absolute value.
Let ${\bf V}$ be an affine $k'$-variety and denote by ${\bf U}$ its restriction of scalar to $k$.
Endow the spaces ${\bf V}(k')$ and ${\bf U}(k)$ with the corresponding analytic topologies and consider the natural maps
${\bf U}(k)\to {\bf U}(k')$ and ${\bf U}(k') \to {\bf V}(k')$.
Then the composed map ${\bf U}(k)\to {\bf V}(k')$ is a homeomorphism.
\end{lemma}

\begin{proof}
By the functoriality of the restriction of scalars and its compatibility with products, it is enough to prove the lemma for ${\bf V}=\mathbb{A}^1$, the 1-dimensional affine space.
In this case we get ${\bf V}(k')\simeq k'$ and ${\bf U}(k)\simeq k^{[k':k]}$ and the composed map $k^{[k':k]} \to k'$, which is a $k$-vector spaces isomorphism, is indeed a homeomorphism by \cite[Theorem 4.6]{BDL}.
\end{proof}

\begin{prop} \label{prop:extscal}
Let $k$ be a field and let $\mathbf{G}$ be a connected, adjoint, $k$-isotropic, $k$-simple $k$-algebraic group.
Then there exists a finite field extension $k'$ of $k$, unique up to equivalence of extensions,
and there exists a connected, adjoint, $k'$-isotropic, absolutely simple $k'$-algebraic group $\mathbf{G}'$, unique up to $k'$-isomorphism, such that ${\bf G}$ is $k$-isomorphic to the restriction of scalars of ${\bf G}'$ from $k'$ to $k$.
In particular, the natural $k'$-morphism ${\bf G}\to {\bf G}'$ gives rise to a group isomorphism ${\bf G}(k)\simeq {\bf G}'(k')$.

If $k$ is endowed with a complete absolute value, endowing $k'$ with the extended absolute value, the isomorphism ${\bf G}(k)\simeq {\bf G}'(k')$ is also a homeomorphism with respect to the corresponding analytic topologies.
\end{prop}

Note that the above proposition is trivial in case ${\bf G}$ is absolutely simple to begin with, in which case we have that $k'=k$ and ${\bf G}'={\bf G}$.

\begin{proof}
The existence of the field $k'$ and the group ${\bf G}'$ follows from the discussion in \cite[I.1.7]{margulis-book}.
Their uniqueness follows from the uniqueness statement in \cite[Theorem I.1.8]{margulis-book}.
The last part follows from Lemma~\ref{lem:extscal}.
\end{proof}

Given a filed $\ell$ and an $\ell$-algebraic group ${\bf H}$ we denote by ${\bf H}(\ell)^+$ the subgroup of ${\bf H}(\ell)$ generated by all the groups of $\ell$-points of the unipotent radicals of all parabolic $\ell$-subgroups of ${\bf H}$, see \cite[I.1.5.2]{margulis-book}.

\begin{prop} \label{prop:explain}
Let $\ell$ and $k$ be fields endowed with complete absolute values.
Let $\mathbf{H}$ be a connected, semi-simple algebraic group defined over $\ell$.
Assume that ${\bf H}$ has no $\ell$-anisotropic factors. 
Let $H$ be an intermediate closed subgroup ${\bf H}(\ell)^+<H<{\bf H}(\ell)$.
Let $\mathbf{G}$ be a connected, adjoint, absolutely simple algebraic group defined over $k$.
Let $\theta:H \to {\bf G}(k)$ be a group homomorphism which is continuous with respect to the corresponding analytic topologies and its image is Zariski dense in ${\bf G}$.
Then there exists a unique field embedding $i:\ell\to k$, which is continuous, and a corresponding unique $k$-algebraic groups morphism ${\bf H}\to {\bf G}$
such that $\theta$ coincides with the precomposition of the corresponding map ${\bf H}(k)\to {\bf G}(k)$ with the injection $H<{\bf H}(\ell)\to {\bf H}(k)$.

Furthermore, if $\ell$ is a local field then the assumption that ${\bf H}$ has no $\ell$-anisotropic factors could be replaced by the assumption that the image of $\theta$ is unbounded in ${\bf G}(k)$.
\end{prop}

\begin{proof}
The first paragraph is proven in \cite[Proposition VII.5.3(a)]{margulis-book}.
Note that in this proof the fields are assumed to be local fields, but this assumption 
is used in the proof only via
\cite[Remark I.1.8.2(IIIa)]{margulis-book}, which applies equally well for complete valued fields.

Assume now that $\ell$ is a local field and $\theta(H)$ is unbounded in ${\bf G}(k)$.
We first remark that the field embedding $i:\ell\to k$, if exists, is unique as it must coincide with the corresponding unique field embedding
we get by replacing ${\bf H}$ by the group ${\bf H}_i$ given by the almost direct product of all $\ell$-istotropic simple factors of ${\bf H}$,
replacing $H$ by the intermediate closed group ${\bf H}_i(\ell)^+< H\cap {\bf H}_i(\ell)<{\bf H}_i(\ell)$ and replacing $\theta$ by its restriction to $H\cap {\bf H}_i(\ell)$,
noting that indeed $H\cap {\bf H}_i(\ell)$ contains ${\bf H}_i(\ell)^+$ by \cite[I.1.5.4(iv)]{margulis-book}.
Second, we remark that the corresponding $k$-algebraic groups morphism ${\bf H}\to {\bf G}$, if exists, is unique.
To see this, we will assume $i$ is given and $\alpha,\beta:{\bf H}\to {\bf G}$ are two such corresponding $k$-morphisms and argue to show that $\alpha=\beta$.
By definition, $\alpha|_H=\beta|_H$. 
By \cite[Proposition I.2.3.1(b)]{margulis-book}, $H^+$ is closed cocompact normal subgroup in ${\bf H}(\ell)$.
It follows that $H$ is a cofinite volume subgroup of ${\bf H}(\ell)$.
Applying \cite[Theorem 6.3]{BF-products} we get that  $\alpha|_{{\bf H}(\ell)}=\beta|_{{\bf H}(\ell)}$
(note that in this reference we regard lattices, but the proof applies equally well to closed subgroups of cofinite volume).
By \cite[Corollary 18.3]{borel} we have that ${\bf H}(\ell)$ is Zariski dense in ${\bf H}$, thus indeed $\alpha=\beta$.
In the sequel we will argue to show the existence of the 
field embedding $i:\ell\to k$ and the corresponding $k$-algebraic groups morphism ${\bf H}\to {\bf G}$.

We let ${\bf H}_a$ be the almost direct products of all $\ell$-anisotropic factors of ${\bf H}$ and denote $H_a={\bf H}_a(\ell)\cap H$.
We also consider the group $H^+={\bf H}(\ell)^+$.
Note that $H_a$ and $H^+$ are commuting normal subgroups of $H$, 
thus 
the Zariski closures of the images of these groups under $\theta$ are commuting algebraic normal subgroups of the simple group ${\bf G}$.
As $H^+$ is cocompact in $H$, its image under $\theta$ is unbounded and in particular non-trivial.
It follows that $\theta(H^+)$ is Zariski dense in ${\bf G}$, and thus $\theta(H_a)$ is central, hence trivial.
We set ${\bf H}'={\bf H}/{\bf H}_a$ and let $H'$ be the image of $H$ in ${\bf H}'$ under the quotient map ${\bf H}(\ell) \to {\bf H}/{\bf H}_a(\ell)$.
We get that $\theta$ factors via $\theta':H'\to {\bf G}(k)$.
Note that ${\bf H}'$ has no $\ell$-anisotroic factors and that $H'$ contains the group ${\bf H}'(\ell)^+$,
by \cite[I.1.5.4(iv) and I1.5.5]{margulis-book}.
We conclude having a continuous $i:\ell\to k$ and a corresponding $k$-algebraic group morphism ${\bf H}'\to {\bf G}$
such that $\theta'$ coincides with the precomposition of the corresponding map ${\bf H}'(k)\to {\bf G}(k)$ with the injection $H'<{\bf H}'(\ell)\to {\bf H}'(k)$.
We are done by considering the composed $k$-morphism ${\bf H}\to {\bf H}' \to {\bf G}$,
noting that indeed $\theta$ coincides with the precomposition of the corresponding map ${\bf H}(k)\to {\bf G}(k)$ with the injection $H<{\bf H}(\ell)\to {\bf H}(k)$.
\end{proof}

Combining Proposition~\ref{prop:explain} with Proposition~\ref{prop:extscal} we readily get the following result, in which the target group ${\bf G}$ is assumed to be $k$-simple rather than absolutely simple.

\begin{cor} \label{cor:cont}
Let $\ell$ and $k$ be fields endowed with complete absolute values.
Let $\mathbf{H}$ be a connected, semi-simple algebraic group defined over $\ell$.
Assume that ${\bf H}$ has no $\ell$-anisotropic factors. 
Let $H$ be an intermediate closed subgroup ${\bf H}(\ell)^+<H<{\bf H}(\ell)$.
Let $\mathbf{G}$ be a connected, adjoint, $k$-simple algebraic group defined over $k$.
Let $\theta:H \to {\bf G}(k)$ be a group homomorphism which is continuous with respect to the corresponding analytic topologies and its image is Zariski dense in ${\bf G}$.
Then there exists a group homomorphism $\hat\theta:{\bf H}(\ell)\to {\bf G}(k)$ which is continuous with respect to the corresponding analytic topologies
such that $\theta=\hat\theta|_H$.

Furthermore, if $\ell$ is a local field then the assumption that ${\bf H}$ has no $\ell$-anisotropic factors could be replaced by the assumption that the image of $\theta$ is unbounded in ${\bf G}(k)$.
\end{cor}


\section{Deducing Theorem~\ref{thm:margulis} from Theorem~\ref{thm:main}} \label{sec:margulis}

We let $k$ be a complete valued field and let $\mathbf{G}$ be a
connected, adjoint, $k$-simple $k$-algebraic group.
Assuming having $\Gamma<H$ and $\rho:\Gamma\to {\bf G}(k)$ as in the theorem,
by \cite[Lemma 6.3]{BF-products} we have that a continuous homomorphism $\hat{\rho}:H \to G$ such that $\hat{\rho}|_\Gamma=\rho$, if exists, is uniquely given.
Below we will prove its existence.

We fix a local field $\ell$.
We let $\mathcal{H}$ be the collection of all connected semi-simple $\ell$-algebraic groups $\mathbf{H}$
which satisfy the theorem, that is: for every lattice $\Gamma$ in $H=\mathbf{H}(\ell)$, whose
projection modulo $\mathbf{N}(\ell)$ for each proper normal $\ell$-isotropic subgroup $\mathbf{N}\lhd \mathbf{H}$ is non-discrete,
and every
homomorphism $\rho:\Gamma \to G$ with unbounded and Zariski dense image $\rho(\Gamma)$ in $\mathbf{G}$,  
there exists a continuous homomorphism $\hat\rho:H\to G$
satisfying $\rho=\hat\rho|_{\Gamma}$.
We argue to show that $\mathcal{H}$ contains all connected semi-simple $\ell$-algebraic groups of $\ell$-rank at least two.

Given a connected semi-simple $\ell$-algebraic group $\mathbf{H}$, 
let us denote by $\bar{\mathbf{H}}$ the associated adjoint group and by $p:\mathbf{H}\to \bar{\mathbf{H}}$ the corresponding $\ell$-isogeny.
The group $\bar{\mathbf{H}}$ is $\ell$-isomorphic to the product of its factors.
Let us denote by $\bar{\mathbf{H}}_0$ the product of all $\ell$-isotropic factors of $\bar{\mathbf{H}}$ 
and let $q:\bar{\mathbf{H}}\to \bar{\mathbf{H}}_0$ be the corresponding projection.
We claim that if $\bar{\mathbf{H}}_0$ is in $\mathcal{H}$, then so is $\mathbf{H}$.
First note that both maps $p:\mathbf{H}(\ell)\to \bar{\mathbf{H}}(\ell)$ and $q:\bar{\mathbf{H}}(\ell)\to \bar{\mathbf{H}}_0(\ell)$ 
have compact kernels and closed cocompact images.
This is clear for the projection $q$, and for $p$ is follows from \cite[Proposition I.2.3.3(i)]{margulis-book}.
In particular, we get that $q\circ p:\mathbf{H}(\ell)\to \bar{\mathbf{H}}_0(\ell)$ has a compact kernel and a closed cocompact image.
Thus, if $\Gamma<\mathbf{H}(\ell)$ is a lattice, then $q\circ p(\Gamma)<\bar{\mathbf{H}}_0(\ell)$ is a lattice, 
and $\Lambda=\ker q\circ p|_\Gamma$ is a finite normal subgroup of $\Gamma$.
Note that for each proper normal $\ell$-isotropic subgroup $\mathbf{N}\lhd \bar{\mathbf{H}}_0$, the projection of $q\circ p(\Gamma)$ modulo
$\mathbf{N}(\ell)$  is non-discrete,
if the corresponding assumption applies to $\Gamma<\mathbf{H}(\ell)$.
Note also that $\Lambda$ is in the kernel of any Zariski dense homomorphism $\rho:\Gamma\to G$, by \cite[Theorem I.1.5.6(i)]{margulis-book},
as $\rho(\Lambda)$ is a finite normal subgroup in the adjoint group $\mathbf{G}$.
It follows that $\rho$ factors through $q\circ p(\Gamma)$, and this proves our claim.
Noting that the groups $\mathbf{H}$ and $\bar{\mathbf{H}}_0$ have equal $\ell$-ranks,
we are left to show that $\mathcal{H}$ contains all connected semi-simple $\ell$-algebraic groups of $\ell$-rank at least two 
which are adjoint and have no $\ell$-anisotropic factors. We proceed to do so.

Assume first that $\mathbf{H}$ is adjoint and it has at least two $k$-isotropic factors.
Up to replacing $\mathbf{H}$ by an $\ell$-isomorphic copy, we assume that $\mathbf{H}=\mathbf{H}_1\times \mathbf{H}_2$,
where the groups $\mathbf{H}_1$ and $\mathbf{H}_2$ are connected, semi-simple, isotropic, adjoint $\ell$-algebraic groups.
We fix a lattice $\Gamma$ in $H=\mathbf{H}(\ell)$
which projection modulo
$\mathbf{N}(\ell)$ for each proper normal $\ell$-isotropic subgroup $\mathbf{N}\lhd \bar{\mathbf{H}}_0$ is non-discrete,
and a homomorphism $\rho:\Gamma \to G$ such that $\rho(\Gamma)$ is unbounded and Zariski dense in $\mathbf{G}$,
and argue to show that
there exists a continuous homomorphism $\hat\rho:H\to G$
satisfying $\rho=\hat\rho|_{\Gamma}$.
We let $\pr_i:\mathbf{H}\to \mathbf{H}_i$ be the corresponding projections, denote by $H'_i$ the closure of $\pr_i(\Gamma)$ in $\mathbf{H}_i(\ell)$
and let $H'=H'_1\times H'_2$.
Therefore $\Gamma<H'$ is a lattice with dense projections in the sense of \cite[Definition 1.1]{BF-products},
thus by \cite[Theorem 1.2]{BF-products} there exists a continuous homomorphism $\hat\rho':H'\to G$
satisfying $\rho=\hat\rho'|_{\Gamma}$.
By \cite[Theorem~II.6.7(a)]{margulis-book} we have that ${\bf H}(\ell)^+<H'<{\bf H}(\ell)$
and by Corollary~\ref{cor:cont},
there exists a continuous homomorphism $\hat\rho:H\to G$
satisfying $\hat\rho'=\hat\rho|_{H'}$. We conclude that $\rho=\hat\rho|_{\Gamma}$,
which finishes the proof in this case.

We are left with the case that $\mathbf{H}$ is $\ell$-simple and of $\ell$-rank at least two.
We fix a lattice $\Gamma$ in $H=\mathbf{H}(\ell)$ and note that by \cite[Theorem III.5.7(b)]{margulis-book}
$\Gamma$ has a finite abelianization.
We set $H^+={\bf H}(\ell)^+$, $\Gamma^+=\Gamma\cap H^+$ and conclude by \cite[Theorem I.2.3.1(c)]{margulis-book} that $\Gamma^+<\Gamma$ is of finite index.
In particular, $\Gamma^+<H$ is a lattice, thus $\Gamma^+$ is also a lattice in the closed subgroup $H^+$ in which it is contained.
We now set $S=H^+$ and note that it satisfies conditions (A) and (B):
condition (B) follows from the assumption of higher rank and condition (A)
follows from \cite[Theorem 6.1]{Bader-Gelander}.
By Theorem~\ref{thm:main} there exists a continuous homomorphism $\hat\rho^+:H^+\to G$ such that $\rho|_{\Gamma^+}=\hat{\rho}^+|_{\Gamma^+}$.
By Corollary~\ref{cor:cont} we get a continuous homomorphism $\hat\rho:H\to G$ such that $\hat\rho|_{H^+}=\hat\rho^+$.
Considering $\Gamma^+$ as a lattice in $\Gamma$ and noting that $\hat\rho|_\Gamma,\rho:\Gamma\to G$ 
are two homomorphisms that coincide on $\Gamma^+$, we deduce by \cite[Lemma 6.3]{BF-products} that
$\hat\rho|_\Gamma=\rho$. This finishes the proof.


\section{A fine version of Theorem~\ref{thm:margulis}} \label{sec:finer}

Taking $\theta=\hat\rho$ in Proposition~\ref{prop:explain} in case the target group ${\bf G}$ is absolutely simple and using Proposition~\ref{prop:extscal} otherwise,
we get the following corollary of Theorem~\ref{thm:margulis}, which could be viewed as a finer version of this theorem.

\begin{cor} \label{cor:finer}
In the setting of Theorem~\ref{thm:margulis},
if ${\bf G}$ is absolutely simple then 
there exist a unique field embedding $i:\ell\to k$, which is continuous, and a corresponding unique $k$-algebraic groups morphism ${\bf H}\to {\bf G}$
such that $\rho$ coincides with the precomposition of the corresponding map ${\bf H}(k)\to {\bf G}(k)$ with the injection $\Gamma<{\bf H}(\ell)\to {\bf H}(k)$.

In the general case, where ${\bf G}$ is merely $k$-simple,
considering the finite field extension $k'$ of $k$
and the $k'$-algebraic group $\mathbf{G}'$ given in Proposition~\ref{prop:extscal},
there exist a unique field embedding $i:\ell\to k'$ which is continuous and a corresponding unique $k$-algebraic group morphism ${\bf H}\to {\bf G}'$
such that $\rho$ coincides with the composition of the corresponding maps $\Gamma<{\bf H}(\ell)$, ${\bf H}(\ell)\to {\bf H}(k')$, ${\bf H}(k')\to {\bf G}'(k')$ and ${\bf G}'(k')\simeq {\bf G}(k)$.
\end{cor}


\end{document}